\documentclass[12pt,a4paper]{article}
\usepackage[latin1]{inputenc}
\usepackage{amsmath}
\usepackage{amsthm}
\usepackage{amsfonts}
\usepackage{amssymb}
\usepackage{graphics}

\usepackage{amsmath,amsthm,amsfonts, amssymb,graphicx, multicol, array}
\usepackage{enumerate}
\usepackage{enumitem}
\usepackage{xcolor}
\usepackage{currfile}

\usepackage[bookmarks,colorlinks,breaklinks, pagebackref]{hyperref}

\hypersetup{linkcolor=blue,citecolor=red,filecolor=blue,urlcolor=blue} 

\usepackage{tabto}

\DeclareMathOperator{\li}{li}

\newtheorem{thm}{Theorem}[section]
\newtheorem{lem}{Lemma}[section]

\newtheorem{exa}{Example}[section]

\newtheorem{dfn}{Definition}[section]

\newcommand{\N}{\mathbb{N}}

\newcommand{\R}{\mathbb{R}}
\newcommand{\C}{\mathbb{C}}

\numberwithin{equation}{section}

\usepackage{fancyhdr}
\title{Equidistributions of Sign Patterns of the Liouville Function and Normal Numbers}
\date{}
\author{N. A. Carella}

\begin{document}
	
\maketitle

\begin{abstract} 
The equidistribution of the double sign patterns of the Liouville function $\lambda$ is proved unconditionally. As application, it is shown that the computable real number $$\sum_{n \geq 1} \frac{1+\lambda(n)}{2^{n}}$$ is a simply normal number in base 4. 
\end{abstract}

\tableofcontents

\section{Introduction} \label{S2266N}
Let $t\ne0$ be a fixed integer and let
\begin{equation}\label{key}
	 (\lambda(n)=\pm1,\lambda(n+t) =\pm1)=(\pm,\pm)
\end{equation} 
 denotes the corresponding double-sign patterns of the Liouville function $\lambda:\mathbb{N} \longrightarrow \{-1,1\}$, defined in \eqref{eq8833LN.050}. This note proposes the following equidistribution results for the double-sign patterns, these include new results, and simpler proofs of the current available results in the literature. 

\begin{thm}\label{thm2266LS.500} Let $x$ be a large number, and let $t\ne0$ be a fixed integer. Then, the double-sign patterns $++$, $+-$, $-+$, and $--$ of the Liouville pair $\lambda(n), \lambda(n+t)$ are equidistributed on the interval $[1,x]$. In particular, each double-sign pattern has the natural density 
	\begin{equation}\label{eq2266LS.510}
		\delta_{\lambda}^{\pm\pm}(t)= \frac{1}{4}.\nonumber
	\end{equation}
\end{thm}

The application to normal numbers is the following.

\begin{thm} \label{thm2266LN.850D} Let $\lambda:\mathbb{N} \longrightarrow \{-1,1\}$ be the Liouville function. Then, the computable real number
\begin{equation} \label{eq2266N.850D}
\sum_{n \geq 1} \frac{1+\lambda(n)}{2^n}=1.16232463762392978595979733583622409170
\ldots , 
\end{equation}	
is a simply normal number in base $b=4$. In particular, the $4$-adic expansion
\begin{equation} \label{eq2266N.850Q}
	\sum_{n \geq 0} \frac{1+\lambda(n)}{4^n}=1.02212032012320013232002110332223010021
	\ldots , 
\end{equation}	
contains infinitely many digit $0$, infinitely many digit $1$, infinitely many digit $2$, and infinitely many digit $3$.
\end{thm}

The essential foundational materials are covered in Section \ref{S8833} to Section \ref{S8844LN}, and the Appendix \ref{S2222A}. The proof of the equidistribution of the double sign patterns of the Liouville function in Theorem \ref{thm2266LS.500} appears in Section \ref{eq6677A}. The proof Theorem \ref{thm2266LN.850D} appears in Section \ref{eq6677B}.\\

\section{Sign Patterns Characteristic Functions for the Liouville Function} \label{S8833}
The formulas for the characteristic functions of several sign-patterns of the Liouville function are developed in this section. 

\subsection{Single-Sign Patterns-Liouville Function }
The Liouville function $\lambda:\mathbb{N} \longrightarrow \{-1,1\}$ is defined by
\begin{equation}\label{eq8833LN.050}
	\lambda(n) =
	(-1)^{v_1+v_2+ \cdots +v_w}     ,
\end{equation}
where $n=p_1^{v_1} p_2^{v_2} \cdots p_w^{v_w}$, the $p_i\geq 2$ are primes, and $v_i\geq1$ are integers. \\

\begin{lem}\label{lem8833LN.200A} If $\lambda(n)\in \{-1,1\}$ is the Liouville function, then, 
\begin{equation}\label{eq8833LN.100}
	\lambda^{\pm}(n)=\left( \frac{1\pm\lambda(n)}{2}\right)= 	\begin{cases}
		1 &\text{\normalfont if } \lambda(n)=\pm1,\\
		0 &\text{\normalfont if } \lambda(n)\ne\pm1,
	\end{cases}
\end{equation}
of the subset of integers
	\begin{equation}\label{eq8833LN.210A}
		\mathcal{N}_{\lambda}^{\pm }	=	\{n\geq 1: \lambda(n)=\pm1\}.
	\end{equation} 
\end{lem}

\subsection{Double-Sign Patterns-Liouville Function }
The analysis of single-sign pattern characteristic functions is extended here to the double-sign patterns
\begin{equation}\label{eq8855LN.100}
	\lambda(n)=\pm1 \quad \text{ and }\quad 	\lambda(n+t)=\pm1, 
\end{equation}
where $t\ne0$ is a small integer, and $n\geq1$ is an integer. Some of the research on multiple sign patterns appear in \cite{HA1986}, \cite{KP1986}, \cite[Corollary 1.7]{TT2015}, \cite{MT2015}, \cite{SA2022}, and similar literature.

\begin{lem}\label{lem8855LN.200} Let $t\ne0$ be small fixed integer, and let $\lambda(n)\in \{-1,1\}$ be the Liouville function. Then,
	\begin{eqnarray}\label{eq8855LN.200}
		\lambda^{\pm \pm}(t)&=&\left( \frac{1\pm \lambda(n)}{2}\right)\left( \frac{1\pm\lambda(n+t)}{2}\right)\\	&=&
		\begin{cases}
			1 &\text{\normalfont if } \lambda(n)=\pm1,\mu(n+t)=\pm1,\\
			0 &\text{\normalfont if } \lambda(n)\ne\pm1,\mu(n+t)\ne\pm1,\nonumber\\
		\end{cases}
	\end{eqnarray}
	are the characteristic functions of the subset of integers
	\begin{equation}\label{eq8855LN.210}
		\mathcal{N}_{\lambda}^{\pm \pm}(t)	=	\{n\geq 1: \lambda(n)=\pm1, \lambda(n+t)=\pm1\}.
	\end{equation} 
\end{lem}

\subsection{$k$-Signs Patterns Liouville Characteristic Functions} \label{S5151K}
The general sign patterns of a vector of values of Liouville functions is consider here. A slight change in notation to simplify the formulas is introduced in this subsection. \\

Let $k\geq1$ be an integer. Define the integer $k$-tuple
\begin{equation}\label{eq5151K.300L}
	\textbf{\textit{a}}=(a_1,a_2, \ldots ,a_k),
\end{equation}
where $0\leq a_1<a_2<\cdots<a_k\leq x$, and the $k$-sign pattern
\begin{equation}\label{eq5151K.310L}
	\boldsymbol{\epsilon}=(\epsilon_1,\epsilon_2,\ldots ,\epsilon_k),
\end{equation} 
where $\epsilon_i\in \{-1,1\}$. The same principle used for single-sign and double-sign patterns is extended to the general $k$-sign patterns characteristic function of $k$-tuple of Liouville function values
\begin{equation}\label{eq5151K.330L}
	(\lambda(n+a_1), \lambda(n+a_2), \ldots ,\lambda(n+a_k))=(\epsilon_1,\epsilon_2,\ldots ,\epsilon_k).
\end{equation}

\begin{lem}\label{lem5151K.200L} Let $n\in \N$ be an integer, and let $\lambda(n)\in \{-1,1\}$ be the Liouville function. If $\textbf{\textit{a}}$ is an integer $k$-tuple, and $\boldsymbol{\epsilon}$ is a $k$-sign pattern, then,
	\begin{eqnarray}\label{eq5151K.200L}
		\lambda(\textbf{a},\boldsymbol{\epsilon},n)&=&\prod_{0\leq i<k}\left( \frac{1\pm\lambda(n+a_i)}{2}\right)\\	&=&
		\begin{cases}
			1 &\text{ \normalfont if } \lambda(n+a_1)=\epsilon_1,\ldots,\,\lambda(n+a_k)=\epsilon_k,\\
			0 &\text{ \normalfont if } \lambda(n+a_1)\ne\epsilon_1,\ldots,\,\lambda(n+a_k)\ne\epsilon_k,\nonumber\\
		\end{cases}
	\end{eqnarray} 
	is the characteristic functions of the subset of integers
	\begin{equation}\label{eq5151K.210L}
		\mathcal{N}_{\lambda}(\textbf{a},\boldsymbol{\epsilon})	=	\{n\geq 1: \lambda(n)=\epsilon_1,\ldots\,\lambda(n+k-1)=\epsilon_k\}.
	\end{equation} 
\end{lem}

\section{Sign Patterns Counting Functions for the Liouville Function} \label{S8844LN}
The analysis of the single-sign patterns for the Liouville and Mobius functions are well known, but are included here as a reference.
\subsection{Single-Sign Patterns-Liouville Counting Function}
The single-sign pattern counting functions $\mathcal{N}_{\lambda}^{+}(x)=\#\{n\leq x:\lambda(n)=1\}$ and $\mathcal{N}_{\lambda}^{+}(x)=\#\{n\leq x:\lambda(n)=-1\}$ of the Liouville function over the integers have the simplest analysis.

\begin{lem}\label{lem8849LN.200} If $\lambda(n)\in \{-1,1\}$ is the Liouville function, then,
	\begin{equation}\label{eq8849LN.200A}
	\mathcal{N}_{\lambda}^{+}(x)=\frac{x}{2}+O\left( xe^{-c\sqrt{\log x}}\right),\nonumber
	\end{equation}
	and 
	\begin{equation}\label{eq8849LN.200B}
	\mathcal{N}_{\lambda}^{-}(x) =\frac{x}{2}+O\left(xe^{-c\sqrt{\log x}}\right), \nonumber
	\end{equation}
\end{lem}
where $c>0$ is an absolute constant.
\begin{proof}[\textbf{Proof}] Utilize the characteristic function, Lemma \ref{lem8833LN.200A}, to express the counting function in the form
	\begin{eqnarray}\label{eq8844LN.100A}
		\mathcal{N}_{\lambda}^{+}(x)&=&\sum_{\substack{n\leq x\\ \lambda(n)=1}}1\\
		&=&\sum_{n\leq x}\left( \frac{1+\lambda(n)}{2}\right)\nonumber \\ &=&\frac{x}{2}+O\left( xe^{-c\sqrt{\log x}}\right)\nonumber, 
	\end{eqnarray}
respectively. The error terms follow from Theorem \ref{thm2222LN.500}. The other case has a similar proof.
\end{proof}
In terms of the single-sign pattern counting functions, the summatory function has the asymptotic formula
\begin{equation}\label{eq8844LN.100C}
	\mathcal{N}_{\lambda}(x)=\sum_{n\leq x}\lambda(n)=\mathcal{N}_{\lambda}^{+}(x)-\mathcal{N}_{\lambda}^{-}(x)=O\left( xe^{-c\sqrt{\log x}}\right). 
\end{equation}
Basically, it is a different form of the Prime Number Theorem
\begin{equation}\label{eq8844LN.100D}
	\pi(x)=\li(x)+O\left(xe^{-c\sqrt{\log x}} \right) ,	
\end{equation}
where $\li(x)=\int_2^2(\log t)^{-1}dt$ is the logarithm integral, and $c>0$ is an absolute constant, see \cite[Eq.~27.12.5]{DLMF}, \cite[Theorem 3.10]{EL1985}, et alii. \\

\subsection{Double-Sign Patterns-Liouville Counting Function}
The counting functions for the single-sign patterns $\lambda(n)=1$ and $\lambda(n)=-1$ are extended to the counting functions for the double-sign patterns 
\begin{equation}\label{eq8877LN.100}
	(\lambda(n),\lambda(n+t))=(\pm1,\pm1).
\end{equation}

The double-sign patterns counting functions are defined by
\begin{equation}\label{eq8877LN.110A}
	\mathcal{N}_{\lambda}^{++}(t,x)=\sum_{\substack{n\leq x\\ \lambda(n)=1,\; \lambda(n+t)=1}}1=\sum_{\substack{n\leq x\\ n\in \mathcal{N}_{\lambda}^{++}(t)}}1,
\end{equation}
\begin{equation}\label{eq8877LN.110B}
	\mathcal{N}_{\lambda}^{+-}(t,x)=\sum_{\substack{n\leq x\\ \lambda(n)=1,\; \lambda(n+t)=-1}}1=\sum_{\substack{n\leq x\\ n\in \mathcal{N}_{\lambda}^{+-}(t)}}1,
\end{equation}
\begin{equation}\label{eq8877LN.110C}
	\mathcal{N}_{\lambda}^{-+}(t,x)=\sum_{\substack{n\leq x\\ \lambda(n)=-1,\; \lambda(n+t)=1}}1=\sum_{\substack{n\leq x\\ p\in \mathcal{N}_{\lambda}^{-+}(t)}}1,
\end{equation}
\begin{equation}\label{eq8877LN.110D}
	\mathcal{N}_{\lambda}^{- -}(t,x)=\sum_{\substack{n\leq x\\ \lambda(n)=-1,\; \lambda(n+t)=-1}}1=\sum_{\substack{n\leq x\\ p\in \mathcal{N}_{\lambda}^{- -}(t)}}1.
\end{equation}

The double-sign patterns counting functions \eqref{eq8877LN.110A} to \eqref{eq8877LN.110D} are precisely the counting functions of the subsets of integers
\begin{multicols}{2}
	\begin{enumerate}
		\item $\mathcal{N}_{\lambda}^{++}(t)\subset \N$ ,
		\item $\mathcal{N}_{\lambda}^{+-}(t)\subset \N$ ,
		\item $\mathcal{N}_{\lambda}^{-+}(t)\subset \N$ ,
		\item $\mathcal{N}_{\lambda}^{--}(t)\subset \N$ ,
	\end{enumerate}
\end{multicols}
defined in \eqref{eq8855LN.210}. 

\begin{lem}\label{lem8877LN.300} Let $x$ be a large number, and let $t\ne0$ be a fixed integer. If $\lambda: \mathbb{Z} \longrightarrow \{-1,1\}$ is the Liouville function, then, 
	\begin{equation}\label{eq8877LN.300}
		\mathcal{N}_{\lambda}^{\pm \pm}(t,x)= \frac{x}{4}+O\left(\frac{x}{(\log \log x)^{1/2-\varepsilon}} \right),\nonumber
	\end{equation}
	where $\varepsilon>0$ is an arbitrary small number.
\end{lem}

\begin{proof}[\textbf{Proof}] Without loss in generality, consider the pattern $(\lambda(n),\lambda(n+t))=(+1,+1)$. Now, use Lemma \ref{lem8855LN.200} to express the double-sign pattern counting function as 
	\begin{eqnarray}   \label{eq8877LN.310}
		4\mathcal{N}_{\lambda}^{++}(t,x)&=&\sum_{n\leq x}\lambda^{++}(t,n)\nonumber\\
		&=&\sum_{n\leq x}\left( 1+\lambda(n)\right) \left( 1+\lambda(n+t)\right)\nonumber\\
		&=&\sum_{n\leq x}\left( 1+\lambda(n)+\lambda(n+t)+\lambda(n)\lambda(n+t)\right)\\
		&=&\sum_{n\leq x}1+\sum_{n\leq x}\lambda(n)  +\sum_{n\leq x}\lambda(n+t)+\sum_{n\leq x}\lambda(n)\lambda(n+t)\nonumber\\
		&\geq&0\nonumber.
	\end{eqnarray}
	The first three finite sums on the last line have the following evaluations or estimates.
	\begin{enumerate}
		\item $ \displaystyle \sum_{n\leq x}1=x, $\tabto{10cm} 
		\item $ \displaystyle \sum_{n\leq x}\lambda(n)=O \left (xe^{-c\sqrt{\log x}}\right )$,\tabto{8cm} see Theorem \ref{thm2222LN.500}.\\
		
		\item $ \displaystyle \sum_{n\leq x}\lambda(n+t)=O \left (xe^{-c\sqrt{\log x}}\right )$, \tabto{8cm}see Theorem \ref{thm2222LN.500},\\
		\item $ \displaystyle \sum_{n\leq x}\lambda(n)\lambda(n+t)=O\left(\frac{x}{(\log \log x)^{1/2-\varepsilon}} \right)$, \tabto{8cm}see Theorem \ref{thm7766N.200} ,
	\end{enumerate}
	where $[x]$ is the largest integer function, $c>0$ is an absolute constant, and $\varepsilon>0$. Summing these evaluations or estimates verifies the claim for $\mathcal{N}_{\lambda}^{++}(t,x)\geq0$. The verifications for the next three double-sign pattern counting functions $\mathcal{N}_{\lambda}^{+-}(t,x)\geq0$, $\mathcal{N}_{\lambda}^{-+}(t,x)\geq0$, and $\mathcal{N}_{\lambda}^{--}(t,x)\geq0$ are similar.
\end{proof}

\section{Equidistribution of Liouville Sign Patterns}\label{S5225LN}

\subsection{Equidistribution of Single-Sign Patterns}
The counting functions are defined by
\begin{equation}\label{eq5225LC.200}
	\mathcal{N}_{\lambda}^{+}(x)=\#\{n\leq x: \lambda(n)=1\}
\end{equation}
and \begin{equation}\label{eq5225LC.210}
\mathcal{N}_{\lambda}^{-}(x)=\#\{n\leq x: \lambda(n)=-1\}. 
\end{equation}
The corresponding densities functions are defined by

\begin{equation}\label{eq5225LC.220}
\delta_{\lambda}^{+}=	\lim_{x\to \infty}\frac{\mathcal{N}_{\lambda}^{+}(x)}{x}\quad \text{ and }\quad \delta_{\lambda}^{-}=	\lim_{x\to \infty}\frac{\mathcal{N}_{\lambda}^{-}(x)}{x}.
\end{equation}
\begin{thm}\label{thm5225LC.200} If $x$ is a large number, then the single-sign patterns of the Liouville function are equidistributed on the interval $[1,x]$. Specifically,
\begin{equation}\label{eq5225LC.230}
	\delta_{\lambda}^{+}=\frac{1}{2}\quad \quad \text{ \normalfont and }\quad \quad \delta_{\lambda}^{-}=\frac{1}{2}\nonumber.
\end{equation}
\end{thm}
\begin{proof}[\textbf{Proof}] Utilize Lemma \ref{lem8849LN.200} to evaluate the limits in \eqref{eq5225LC.220}.
\end{proof}
The nontrivial result for the summatory Liouville function
\begin{eqnarray}\label{eq5225LC.240}
	\sum_{n\leq x}\lambda(n)
	&=&\mathcal{N}_{\lambda}^{+}(x)-\mathcal{N}_{\lambda}^{-}(x)=O\left( xe^{-c\sqrt{\log x}}\right)
\end{eqnarray}
has no main term. It vanishes because the number of single-sign patterns $\mathcal{N}_{\lambda}^{+}(x)$ and $\mathcal{N}_{\lambda}^{-}(x)$ have the same main terms, see Lemma \ref{lem8849LN.200}. This is implied by the equidistribution of the single sign patterns.

\subsection{Equidistribution of Double-Sign Patterns}
The equidistribution results for single-sign patterns are extended to equidistribution of the double-sign patterns. \\

Recall that the double sign counting function is defined by
\begin{equation}\label{eq5225LN.500}
\mathcal{N}_{\lambda}^{\pm\pm}(t,x)=\#\{n\leq x:\lambda(n)=\pm1 ,\lambda(n+t)=\pm1\}, 
\end{equation}
and the natural density of a double-sign pattern is defined by
\begin{equation}\label{eq5225LN.510}
	\delta_{\lambda}^{\pm\pm}(t)=  \lim_{x\to \infty}\;	\frac{\#\{n\leq x:\lambda(n)=\pm1 ,\lambda(n+t)=\pm1\}}{x}.
\end{equation}

\begin{proof}[\textbf{Proof of Theorem {\normalfont\ref{thm2266LS.500}}}] For large $x$, the asymptotic formula for $\mathcal{N}_{\lambda}^{\pm\pm}(t,x)>0$ is proved in Lemma \ref{lem8877LN.300}. Next, compute the limit of the proportion of double-sign pattern
	\begin{eqnarray}\label{eq5225.550}
		\delta_{\lambda}^{\pm\pm}(t)&=&  \lim_{x\to \infty}\;	\frac{\#\{n\leq x:\lambda(n)=\pm1 ,\lambda(n+t)=\pm1\}}{x}\nonumber\\
		&=&\lim_{x\to \infty}\;	\frac{[x]+O\left(x(\log \log x)^{-1/2+\varepsilon} \right)}{4x}\nonumber\\
		&=&\frac{1}{4}.\nonumber
	\end{eqnarray} 
	This proves that the double-sign patterns $++$, $+-$, $-+$, and $--$ are equidistributed on the interval $[1,x]$
	as $x\to \infty$.
\end{proof}

\begin{exa}\label{exa5225LN.600}{\normalfont Let $t=1$. By Theorem \ref{thm2266LS.500}, in any sufficiently large interval $[1,x]$, the number of any double-sign pattern $\lambda(n)=\pm1, \lambda(n+1)=\pm1$ is
		\begin{equation}\label{eq5225LN.610}
			\mathcal{N}_{\lambda}^{\pm\pm}(t,x)	=\delta_{\lambda}^{\pm\pm}(t)x+o(x)=\frac{1}{4}x+O\left(\frac{x}{(\log \log x)^{1/2-\varepsilon}} \right).
		\end{equation}		
		The numerical data for $x=10^5$, shows that the actual value of the autocorrelation function is
		\begin{equation}\label{eq5225LN.630}
			\sum_{n\leq x}\lambda(n)\lambda(n+1)
			=68,
		\end{equation}
		and the actual values of the double-sign pattern counting functions are tabulated below.
\begin{center}
\begin{tabular}{ c|c|c|c } 
$\lambda(n)$ & $\lambda(n+1)$ &Actual Count &Expected $\mathcal{N}_{\lambda}^{\pm\pm}(1,x)$\\
\hline 
$+1$ & $+1$ & $99492/4$&$100000/4+o(x)$ \\ 
$+1$ & $-1$ & $99932/4$ &$100000/4+o(x)$\\ 
$-1 $& $+1 $& $99932/4$& $100000/4+o(x)$\\
$-1$ &$ -1 $& $100644/4$&$100000/4+o(x)$ \\		
\end{tabular}
\end{center}
		Given the small scale of this experiment, $x=10^5$, the actual data fits the prediction very well. The tiny differences among the actual values, (in third column), and the prediction by the double-sign pattern counting functions $\mathcal{N}_{\lambda}^{\pm\pm}(1,x)$ seem to be properties of the races between the different subsets of integers $\mathcal{N}_{\lambda}^{\pm \pm}(t)$ attached to the double-sign patterns, see \eqref{eq8855LN.210}. For an introduction to the literature in comparative number theory, prime number races, and similar topics, see \cite{GM2004}, et cetera.
	}
\end{exa}

\section{Applications to Normal Numbers I}\label{eq6677I}
The earliest study was centered on the distributions of the digits in the decimal expansions of algebraic irrational numbers such as $\sqrt{2}=1.414213562 \ldots$. This problem is known as the Borel conjecture. \\

The digits $w_n\in \{0,1,2,\ldots, b-1\}$ in the $b$-adic expansion $\alpha=\sum_{n\geq0}w_nb^{-n}$ of a normal number are random, and uniformly distributed.

\begin{dfn} \label{dfn6677I.100}{\normalfont Let $b>1$ be an integer base. A real number $\alpha \in (0,1)$ is said to be \textit{simply normal in base $b$} if each digit  $w \in \{0,1,2,\ldots, b-1\}$ in its $b$-adic expansion $\alpha=\sum_{n\geq0}w_nb^{-n}$ occurs with probability $$P(w)=\frac{1}{b}.$$ 
	}
\end{dfn}

\begin{dfn} \label{dfn6677I.120}{\normalfont Let $b>1$ be an integer base, and let $k\geq$ be an integer. A real number $\alpha \in (0,1)$ is said to be \textit{normal in base $b$} if every sequence of digits $w=w_0w_1\cdots w_{k-1} \in \{0,1,2,\ldots, b-1\}^k$ occurs in its $b$-adic expansion $\alpha=\sum_{n\geq0}w_nb^{-n}$ with probability $$P(w)=\frac{1}{b^k}.$$ 
	}
\end{dfn}

\subsection{A Result for Simply Normal in Base $b=2$}\label{eq6677A}
The $2$-adic digits of the real number $\beta=0.w_1w_2w_3\cdots$ are generated by the map
\begin{equation}\label{eq6677I.200}
	n\longrightarrow w_n=\frac{1+\lambda(n)}{2} \in \{0,1\}.
\end{equation} 
The single digit pattern counting function is defined by \begin{equation}\label{eq6677I.210}
	\mathcal{N}_{\beta}(a,x)=\#\{w_n=a:n\leq x\},
\end{equation} and the natural digit density is defined by
\begin{equation}\label{eq6677I.220}
	p_{\beta}(a)=\lim_{x\to \infty}\frac{\mathcal{N}_{\beta}(a,x)}{x}=\lim_{x\to \infty}\frac{\mathcal{N}_{\lambda}^{\pm}(x)}{x}
	=\delta_{\lambda}^{\pm}.
\end{equation}
\begin{thm} \label{thm6677LN.200S} Let $\lambda:\mathbb{N} \longrightarrow \{-1,1\}$ be the Liouville function. Then, the computable real number
	\begin{equation} \label{eq6677LN.230}
	\beta=	\sum_{n \geq 1} \frac{1+\lambda(n)}{2^n}=1.16232463762392978595979733583622409170
		\ldots \nonumber, 
	\end{equation}	
	is simply normal number in base $2$.
\end{thm}

\begin{proof}[\textbf{Proof}] By Theorem \ref{thm5225LC.200}, the precise probabilities of the values of the single digit pattern are the followings.
	\begin{enumerate}
		\item $ \displaystyle p_{\beta}(w_n=0)=\delta_{\lambda}^{-}=\frac{1}{2}, $\tabto{6cm} since $0=\frac{1+\lambda(n)}{2}$ with $\lambda(n)=-1$,
		\item $ \displaystyle p_{\beta}(w_n=1)=\delta_{\lambda}^{+}=\frac{1}{2}$,\tabto{6cm} since $1=\frac{1+\lambda(n)}{2}$ with $\lambda(n)=1$.
	\end{enumerate}
Each digit has the same probability
	\begin{equation}\label{eq6677L.240}
		p_{\beta}(w_n=0)=p_{\beta}(w_n=1)=\frac{1}{2},
	\end{equation}
Therefore, by Definition \ref{dfn6677I.100}, the real number $\beta$ is simply normal number.
\end{proof}

\subsection{A Result for Simply Normal in Base $b=4$}\label{eq6677B}
The approach used to prove simply normality in base $b=2$ is here to verify a result for simply normality in base $b=4$. The $4$-adic expansion of the number under analysis has the form
\begin{eqnarray} \label{eq2266N.800}
	\sum_{n \geq 1} \frac{1+\lambda(n)}{2^n}&=&1.16232463762392978595979733583622409170\nonumber\\
	&=&\sum_{n \geq 1} \frac{w_n}{b^n}. 
\end{eqnarray}	

The $n$th digit is defined by
\begin{equation}\label{eq6677L.830}
	w_n=\frac{1+\lambda(n)}{2}+\left( \frac{1+\lambda(n+1)}{2}\right )\cdot2.
\end{equation} 
This is generated by the double sign patterns
\begin{equation}\label{eq6677L.835}
	\left (\frac{1+\lambda(n)}{2}, \frac{1+\lambda(n+1)}{2}\right ).
\end{equation} 
 
\begin{proof}[{\normalfont \textbf{Proof of Theorem \ref{thm2266LN.850D}}}] By Theorem \ref{thm2266LS.500}, the precise probabilities for the double digit patterns are the followings.
	\begin{enumerate}
		\item $(0,0)=\left (\frac{1+\lambda(n)}{2}, \frac{1+\lambda(n+1)}{2}\right ),$ \tabto{7cm}$ \displaystyle p_{\beta}(w_n=0)=\delta_{\lambda}^{--}(1)=\frac{1}{4}, $ \\
		
\item $(0,1)=\left (\frac{1+\lambda(n)}{2}, \frac{1+\lambda(n+1)}{2}\right ),$ \tabto{7cm}$ \displaystyle p_{\beta}(w_n=1)=\delta_{\lambda}^{-+}(1)=\frac{1}{4}, $ \\

\item $(1,0)=\left (\frac{1+\lambda(n)}{2}, \frac{1+\lambda(n+1)}{2}\right ),$ \tabto{7cm}$ \displaystyle p_{\beta}(w_n=2)=\delta_{\lambda}^{+-}(1)=\frac{1}{4}, $ \\

\item $(1,1)=\left (\frac{1+\lambda(n)}{2}, \frac{1+\lambda(n+1)}{2}\right ),$ \tabto{7cm}$ \displaystyle p_{\beta}(w_n=3)=\delta_{\lambda}^{++}(1)=\frac{1}{4}, $ \\
	\end{enumerate}
Each of the 4 double digit patterns \eqref{eq6677L.830}, equivalently each digit $w_n\in \{0,1,2,3\}$, has the same probability
	\begin{equation}\label{eq6677L.840}
		p_{\beta}(w_n=0)=p_{\beta}(w_n=1)=p_{\beta}(w_n=2)=p_{\beta}(w_n=3)=\frac{1}{4}.
	\end{equation}
Therefore, by Definition \ref{dfn6677I.120}, the real number \eqref{eq2266N.800} is simply normal number in base $b=4$.
\end{proof}


\section{Appendix A: Basic Results for the Liouville Function}\label{S2222A}
Some standard results required in the proofs of the equidistributions of the sign patterns of the Liouville function are recorded in this section.

\subsection{Average Orders of Liouville Functions}\label{S2222LN}
\begin{thm} \label{thm2222LN.500} If $\lambda: \N\longrightarrow \{-1,1\}$ is the Liouville function, then, for any large number $x$, the following statements are true.
\begin{enumerate} [font=\normalfont, label=(\roman*)]
		\item $\displaystyle \sum_{n \leq x} \lambda(n)=\zeta(1/2)x^{1/2}+O \left (xe^{-c\sqrt{\log x}}\right )$, \tabto{8cm} unconditionally,
		\item $\displaystyle \sum_{n\leq x}\frac{\lambda(n)}{n}=O\left( e^{-c\sqrt{\log x}} \right ), $ \tabto{8cm} unconditionally,
\end{enumerate}where $c>0$ is an absolute constant.
\end{thm}
The most recent research on the summatory Liouville function seems to be \cite{MT2021}.\\

\subsection{Twisted Exponential Sums}
One of the earliest result for exponential sum with multiplicative coefficients is stated below.

\begin{thm} \label{thm3970LE.300} {\normalfont (\cite{DH1937})} If $\alpha\ne0$ is a real number, and $c>0$ is an arbitrary constant, then
	\begin{enumerate} [font=\normalfont, label=(\roman*)]
		\item $\displaystyle \sup_{\alpha\in\R}\sum_{n \leq x} \lambda(n)e^{i 2 \pi  \alpha n}<\frac{c_3x}{(\log x)^{c}}$, \tabto{8cm} unconditionally,
		\item $\displaystyle \sup_{\alpha\in\R}\sum_{n \leq x} \frac{\lambda(n)}{n}e^{i 2 \pi  \alpha n}<\frac{c_4}{(\log x)^{c}}, $ \tabto{8cm} unconditionally,
	\end{enumerate}
	where $c_3=c_3(c)>0$ and $c_4=c_4(c)>0$ are constants depending on $c$, as the number $x \to \infty$.
\end{thm}

Advanced, and recent works on these exponential sums with multiplicative coefficients, and the more general exponential sums
\begin{equation}\label{S2222FN.100}
	\sum_{n \leq x} f(n)e^{i 2 \pi  \alpha n}
\end{equation}
where $f:\N\longrightarrow \C$ is a function, 
are covered in \cite{MV1977}, \cite{HS1987}, \cite{BH1991}, \cite{MS2002}, et alii.

\subsection{Logarithm Average and Arithmetic Average Connection} \label{S7766N}
The connection between the logarithm average 
\begin{equation} \label{eq7766N.050}
	\sum_{n \leq x} \frac{f(n) }{n} 
\end{equation}
and the arithmetic average 
\begin{equation} \label{eq7766N.060}
	\sum_{n \leq x} f(n) 
\end{equation}
of an arithmetic function $f: \N \longrightarrow \C$ is important in partial summations. The required error term to compute the arithmetic average \eqref{eq7766N.060} directly from the logarithm average \eqref{eq7766N.050} is explained in \cite[Section 2.12]{HA2013}, see also \cite[Exercise 2.12]{HA2013}. \\

\begin{lem}\label{lem7766.400} Let $t\ne0$ be a small integer, and let $x\geq1$ be a large number. If the logarithm average $A(x)=\sum_{n \leq x} \lambda(n) \lambda(n+t)n^{-1}=O(\log x)(\log\log x)^{-1/2}$, then the arithmetic average 
\begin{equation}\label{7766.400}
		\sum_{n \leq x} \lambda(n) \lambda(n+t)<\frac{x}{(\log \log x)^{1/2-\varepsilon}}\nonumber,
	\end{equation}
where $\varepsilon>0$ is a small number. 
\end{lem}
\begin{proof} Assume $B(z)=\sum_{n \leq z} \lambda(n) \lambda(n+t)\geq z(\log \log z)^{-1/2+\varepsilon}$. Then,
	\begin{eqnarray}\label{eq7766.410}
		\frac{\log x}{(\log \log x)^{1/2}}&\gg&\sum_{n \leq x} \frac{\lambda(n) \lambda(n+t)}{n}\\
		&=&\int_1^x \frac{1}{z} \,dB(z)\nonumber \\
		&=&\frac{B(x)}{x}+\int_1^x \frac{B(z)}{z^2}dz\nonumber	.
	\end{eqnarray}
	Since the integral \begin{equation}\label{eq7766.420}
		\int_1^x \frac{B(z)}{z^2}dz=	\int_2^x\frac{1}{z(\log \log z)^{1/2-\varepsilon}} dz\gg \frac{\log x}{(\log \log x)^{1/2-\varepsilon}},
	\end{equation}
	for sufficiently large $x\geq1$, the assumption is false. Hence, it implies that $B(x)< x(\log \log x)^{-1/2+\varepsilon}$.
\end{proof}

A more general result is worked out in the next result.
\begin{lem}\label{lem7766.400G} If $x$ is a large number, and $0\leq a_1<a_2<\ldots< a_k$ is an integer $k$-tuple, then a nontrivial logarithmic average 
	\begin{equation}\label{7766.400A}
		\sum_{n \leq x} \frac{\lambda(n+a_1)\lambda(n+a_2)\cdots \lambda(n+a_k)}{n}=o(\log x)\nonumber.
	\end{equation}
	
implies a nontrivial arithmetic average
	\begin{equation}\label{7766.400B}
		\sum_{n \leq x} \lambda(n+a_1)\lambda(n+a_2)\cdots \lambda(n+a_k)=o(x)\nonumber.
	\end{equation}
\end{lem}
\begin{proof} Suppose that the arithmetic average is trivial, that is,
	\begin{equation}\label{eq7766.405G}
		B(z)=\sum_{n \leq z} \lambda(n+a_1)\lambda(n+a_2)\cdots \lambda(n+a_k)=cz+o(z),
	\end{equation} 
where $c>0$ is a constant. Then,
	\begin{eqnarray}\label{eq7766.410G}
o(\log x)	&=&\sum_{n \leq x} \frac{\lambda(n+a_1)\lambda(n+a_2)\cdots \lambda(n+a_k)}{n}\\
		&=&\int_1^x \frac{1}{z} \,dB(z)\nonumber \\
		&=&\frac{B(x)}{x}+\int_1^x \frac{B(z)}{z^2}dz\nonumber	.
	\end{eqnarray}
Substituting and evaluating the integral yield
	\begin{eqnarray}\label{eq7766.420G}
	o(\log x)		&=&\frac{B(x)}{x}+\int_1^x \frac{B(z)}{z^2}dz\\
&=&c+o(1)+c\log x+o(\log x)	\nonumber,
\end{eqnarray}
for all sufficiently large $x\geq1$.
Clearly, the assumption \eqref{eq7766.405G} is false. Hence, it implies that $B(x)=o(x)$.
\end{proof}
\subsection{Autocorrelation Functions} \label{S474MN}

The current estimate of the logarithmic average order of a product of two shifted Liouville functions. The result has the following asymptotic formula.

\begin{thm} \label{thm7766L.100} Let $\lambda:\mathbb{N} \longrightarrow \{-1,1\}$ be the Liouville function, and let $x$ be a large number. If $t\ne 0$ is a fixed integer, then 
	\begin{equation} \label{eq5757L.100}
\sum_{n \leq x} \frac{\lambda(n) \lambda(n+t)}{n} =O \left (\frac{\log x}{\sqrt{\log \log x} } \right )\nonumber. 
	\end{equation}	
\end{thm}	
The proof appears in Theorem \ref{thm7766N.200} or \cite[Corollary 2]{HH2022}. This improves the estimate $O((\log x)(\log \log \log x)^{-c})$, where $c>0$ is a constant, described in \cite[p. 5]{TT2015}.  

The previous result is sufficient to derive a weak form of the arithmetic average order a product of two shifted Liouville functions.
\begin{thm} \label{thm7766N.200} Let $\lambda:\mathbb{N} \longrightarrow \{-1,1\}$ be the Liouville function, and let $x$ be a large number. If $t\ne 0$ is a fixed integer, then 
	\begin{equation} \label{eq5757L.700}
		\sum_{n \leq x} \lambda(n) \lambda(n+t)=O \left (\frac{\log x}{(\log \log x)^{1/2-\varepsilon}} \right )\nonumber, 
	\end{equation}	
where $\varepsilon>0$ is a small number.
\end{thm}	
\begin{proof}[\textbf{Proof}] This follows from Lemma \ref{lem7766.400}.
\end{proof}

The general estimate of the logarithmic average order of a product of an even number of shifted Liouville functions remains an open problem. However, logarithmic average order of a product of an odd number of shifted Liouville functions has a nontrivial result. 	
\begin{thm} \label{thm7766L.150} {\normalfont (\cite[Theorem 1.1]{TT2018})}.
	Let $k\geq1$ be an odd natural number, and let $a_1,\ldots, a_k, b_1,\ldots, b_k$ be natural
	numbers. Then,
\begin{equation}\label{eq7766.150}
\sum_{n\leq x} \frac{\lambda(a_1n+b_1)\lambda(a_2n+b_2)\cdots\lambda(a_kn+b_k)}{n} = o(\log x) \nonumber,
\end{equation}	
as $x\to \infty$.
\end{thm}

\begin{thm} \label{thm7766L.250} Let $k=2m+1\geq1$ be an odd integer, and let $0\leq a_1,\ldots, a_k$ be an integer $k$-tuple. Then,
\begin{equation}\label{eq7766.250}
		\sum_{n\leq x}\lambda(n+a_1)\lambda(n+a_2)\cdots \lambda(n+a_k) = o(x) \nonumber,
	\end{equation}	
	as $x\to \infty$.
\end{thm}

\begin{proof}[\textbf{Proof}] Set $a_1=a_1=\cdots=a_k=1$, and let $0\leq b_1<b_2<\cdots<b_k$ be an integer $k$-tuple. Then, this follows from Lemma \ref{lem7766.400G} and Theorem \ref{thm7766L.150}.
\end{proof}

\subsection{Distribution Functions }\label{S4422MN}
For any fixed integer $k>4$, the $k$-tuples $\mu(n+a_1), \mu(n+a_1), \ldots,\mu(n+a_k) $ of Mobius values are not random, but pseudorandom or quasirandom. For example, the $k$-tuple 
\begin{equation}\label{eq4422MN.510A}
	-1,-1,-1,-1,\mu(n+a_4), \mu(n+a_5), \ldots,\mu(n+a_k),
\end{equation} 
and infinitely many other similarly structured $k$-tuples are not possible. This property seems to preempt the effect of the Linear Independence Conjecture on the summatory Mobius function. Assuming, the LI, the limits
\begin{equation}\label{eq4422MN.520A}
	\liminf_{x\to \infty}	\frac{\sum_{n \leq x} \mu(n)}{x^{1/2}}=-\infty\quad \text{ and } \quad \liminf_{x\to \infty}	\frac{\sum_{n \leq x} \mu(n)}{x^{1/2}}=\infty
\end{equation}
were proved in \cite{IA1942}, and refinements in \cite{ML1980}. This conjecture seems to imply that $\sum_{n \leq x} \mu(n)=O \left (x^{1/2+\varepsilon}\right )$, where $\varepsilon>0$. But, the Simple Zero Conjecture seems to imply that $\sum_{n \leq x} \mu(n)=O \left (x^{1/2}\right )$, see \cite[Theorem 14.29]{TE1987} for details. The numerical data are given in \cite{KV2003} is not conclusive, and it supports many different conjectures on the Mertens sum.\\

For any fixed integer $k>4$, the $k$-tuples $\lambda(n+a_1), \lambda(n+a_1), \ldots,\lambda(n+a_k) $ of Liouville function values appears to be random, there are no known obstacles. For example, the $k$-tuple \begin{equation}\label{eq4422MN.510B}
	-1,-1,-1,-1,\lambda(n+a_4), \lambda(n+a_5), \ldots,\lambda(n+a_k),
\end{equation} and infinitely many other similarly structured $k$-tuples are possible. Thus, the law of iterated logarithm for sequences of independent random variables seem to imply that
\begin{equation}\label{eq4422MN.520B}
	\liminf_{x\to \infty}	\frac{\sum_{n \leq x} \lambda(n)}{\sqrt{2x\log \log x}}=-\infty\quad \text{ and } \quad \liminf_{x\to \infty}	\frac{\sum_{n \leq x} \lambda(n)}{\sqrt{2x\log \log x}}=\infty.
\end{equation}
In synopsis, the Linear Independence Conjecture does seem to apply to the summatory Liouville function. In particular, $\sum_{n \leq x} \lambda(n)=\zeta(1/2)x^{1/2}+O \left (x^{1/2+\varepsilon}\right )$, where $\varepsilon>0$.\\

The information in \eqref{eq4422MN.510A}, \eqref{eq4422MN.520A}, and the information in \eqref{eq4422MN.510B}, \eqref{eq4422MN.520B} seems to imply that the random or pseudorandom variables 
\begin{equation}\label{eq4422MN.530}
	L(x)=\sum_{n \leq x} \lambda(n) \quad \quad \text{ and }\quad \quad M(x)=\sum_{n \leq x} \mu(n)
\end{equation} 
have different distribution functions.


\currfilename.\\

\begin{thebibliography}{998}

\bibitem{BH1991} Baker, R. C., Harman, G. \textit{\color{red}Exponential sums formed with the Mobius function}. Journal of the London Math. Soc. (2) 43 (1991), 193-198.






\bibitem{DH1937} Davenport, Harold. \textit{\color{red}On some series involving arithmetical functions. II.} Quart. J. Math. Oxf., 8:313-320, 1937.





\bibitem{DLMF} NIST \textit{\color{red}Digital Library of Mathematical Functions}. http://dlmf.nist.gov/, 2019. F. W. J. Olver, ..., and M. A. McClain, eds.

\bibitem{EL1985}  Ellison, William; Ellison, Fern. \textit{\color{red}Prime numbers.} A Wiley-Interscience Publication. John Wiley and Sons, Inc., New York; Hermann, Paris, 1985.



\bibitem{GM2004} Granville, A.; Martin, G. \textit{\color{red}Prime Number Races.} http://arxiv.org/abs/math/0408319.



\bibitem{HA2013} Hildebrand, Adolph. \textit{\color{red}Introduction to Analytic Number Theory Math 531 Lecture Notes, Fall 2005.} http://www.math.uiuc.edu/~hildebr/ant.

\bibitem{HA1986} Hildebrand, Adolph. \textit{\color{red} On consecutive values of the Liouville function}. Enseign. 	Math. (2) 32 (1986), no. 3-4, 219-226.


\bibitem{HH2022} Helfgott, Harald Andres. \textit{\color{red}Expansion, divisibility and parity: an explanation}. http://arxiv.org/abs/2201.00799. 









\bibitem{HR2021} Helfgott, Harald Andres; Radziwill, Maksym.	\textit{\color{red}Expansion, divisibility and parity.} http://arxiv.org/abs/2103.06853.

\bibitem{HS1987} Hajela, D.; Smith, B. \textit{\color{red} On the maximum of an exponential sum of the Mobius function.} Lecture Notes in Mathematics (Springer, Berlin, 1987) 145-164.



\bibitem{IA1942} Ingham, A. E. \textit{\color{red}On two conjectures in the theory of numbers}. Amer. J. Math., 64(1):313-319, 1942.


\bibitem{KP1986} Kaczorowski, J.; Pintz, J. \textit{\color{red}Oscillatory properties of arithmetic functions I}. Acta Math. Acad. Sci. Hung. 48 (1986), 173-185.	


\bibitem{KV2003} Kotnik, T.; van de Lune, J. \textit{\color{red}On the order of the Mertens function.} Experimental Mathematics, 13 473-481, 2003.




Thesis, Leiden University, 2017.



\bibitem{ML1980} Montgomery, H.L. \textit{\color{red}The zeta function and prime numbers}. Proceedings of the Queen's Number Theory Conference, 1979, Queen's Univ., Kingston, Ont., 1980, 1-31.




\bibitem{MS2002} Murty, R.; Sankaranarayanan, A.
\textit{\color{red}Averages of exponential twists of the Liouville function.} Forum
Mathematicum 14 (2002), 273-291. 




\bibitem{MT2021}Martin, Greg; Mossinghoff, Michael J. ; Trudgian, Timothy S. \textit{\color{red}Fake Mu's.} http://arxiv.org/abs/2112.05227. 

\bibitem{MT2015} Matomaki, Kaisa, Radziwill, M., Tao, T. \textit{\color{red}Sign patterns of the Liouville and Mobius functions.} http://arxiv.org/abs/1509.01545.




\bibitem{MV1977} Montgomery, H., Vaughan, R.  C.\textit{\color{red}Exponential sums with multiplicative coefficients.} Inventiones Math. 43 (1977), 69-82.







\bibitem{SA2022} Srinivasan, Anitha. \textit{\color{red} Infinitely many sign changes of the Liouville function}. Proc. Amer. Math. Soc. 150 (2022), 3799-3809. 







\bibitem{TE1987} Titchmarsh, E. C. \textit{\color{red}The Theory of the Riemann Zeta-Function}. Oxford University Press; 2nd edition, 1987.


\bibitem{TT2015} Tao, T. \textit{\color{red} The logarithmically averaged Chowla and Elliott conjectures for two-point correlations}. http://arxiv.org/abs/1509.05422. 

\bibitem{TT2018} Tao, T.; Teravainen, J. {Odd order cases of the logarithmically averaged chowla conjecture.} J. Theor. Nombres Bordeaux 30 no. 3, 997-1015, 2018.






	
	
\end{thebibliography}
\end{document}